\documentclass[12pt]{article}
\usepackage{e-jc}

\usepackage{amsmath}
\usepackage{enumitem}
\usepackage[bottom]{footmisc}
\usepackage{randtext}
\usepackage{tikz-cd}

\dateline{Oct 11, 2020}{Apr 9, 2021}{TBD}
\MSC{05D10, 03H05, 22A20, 54J05, 54D80}
\Copyright{The author. Released under the CC BY license (International 4.0).}
\title{Ramsey theory for layered semigroups}

\author{Jordan Mitchell Barrett\\
\small School of Mathematics and Statistics\\[-0.8ex]
\small Victoria University of Wellington\\[-0.8ex] 
\small New Zealand\\
\small\tt \randomize{math@jmbarrett.nz}}

\setlist[enumerate,1]{label={(\roman*)}}
\setlist[enumerate,2]{label={(\alph*)}}

\newcommand{\N}{\mathbb{N}}			% Natural numbers
\newcommand{\V}{\mathbb{V}}

\newcommand{\F}{\mathcal{F}}
\newcommand{\M}{\mathcal{M}}
\newcommand{\Nn}{\mathcal{N}}
\newcommand{\U}{\mathcal{U}}		% Curly "U" for ultrafilter

\newcommand{\h}{{}^*}				% Nonstandard extension
\newcommand{\kh}[1]{{}^{#1*}}		% k-iterated hyperextension
\newcommand{\g}{{}^\gamma}

\theoremstyle{definition}
\newtheorem{axiom}[theorem]{Axiom}
\newtheorem{axioms}[theorem]{Axioms}

\newcommand{\abs}[1]{{\left\lvert #1 \right\rvert}}		% Absolute value
\newcommand{\ct}{{}^\frown}								% String concatenation
\newcommand{\defeq}{
	\mathrel{\mathrel{\mathop:}\mkern-1.2mu=}}			% := symbol
\newcommand{\FIN}{\mathrm{FIN}}
\newcommand{\floor}[1]{{\left\lfloor #1 \right\rfloor}}		% Floor function
\newcommand{\om}{\overline{\omega}}

\DeclareMathOperator{\dom}{dom}				% Domain
\DeclareMathOperator{\id}{id}				% Identity function
\DeclareMathOperator{\FS}{FS}
\DeclareMathOperator{\Pow}{\mathcal{P}}		% Power set
\DeclareMathOperator{\range}{range}			% Range
\DeclareMathOperator{\supp}{supp}			% Support

\newcommand{\hA}{\h\hspace{-0.25em}A}
\newcommand{\hB}{\h\hspace{-0.15em}B}
\newcommand{\hK}{\h\hspace{-0.15em}K}
\newcommand{\hM}{\h\hspace{-0.15em}M}
\newcommand{\hS}{\h\hspace{-0.1em}S}
\newcommand{\hf}{\h\hspace{-0.2em}f}
\newcommand{\halpha}{\h\hspace{-0.15em}\alpha}
\newcommand{\hbeta}{\h\hspace{-0.15em}\beta}

\newcommand{\gA}{\g\hspace{-0.25em}A}
\newcommand{\gL}{\g\hspace{-0.15em}L}
\newcommand{\gS}{\g\hspace{-0.15em}S}

\newcommand{\khalpha}[1]{\kh{#1}\hspace{-0.15em}\alpha}

\begin{document}

\maketitle
%{\let\thefootnote\relax\footnotetext{Dedicated to the late Ronald L.\ Graham (1935--2020). This paper would never have been written if not for his groundbreaking work in Ramsey theory.}}

\vspace{3mm}
\begin{abstract}
    We further develop the theory of \textit{layered semigroups}, as introduced by Farah, Hindman and McLeod, providing a general framework to prove Ramsey statements about such a semigroup $S$. By nonstandard and topological arguments, we show Ramsey statements on $S$ are implied by the existence of ``coherent'' sequences in $S$. This framework allows us to formalise and prove many results in Ramsey theory, including Gowers' $\FIN_k$ theorem, the Graham--Rothschild theorem, and Hindman's finite sums theorem. Other highlights include: a simple nonstandard proof of the Graham--Rothschild theorem for strong variable words; a nonstandard proof of Bergelson--Blass--Hindman's partition theorem for located variable words, using a result of Carlson, Hindman and Strauss; and a common generalisation of the latter result and Gowers' theorem, which can be proven in our framework.
\end{abstract}
\vspace{7mm}

\section{Introduction}

Ramsey theory mathematically studies to what extent regular configurations appear in disorder. A Ramsey-type result typically has the following form: for any finite colouring of some structure $\M$, we can find a monochromatic substructure $\Nn \subseteq M$ with certain properties. The structure $\M$ and required properties of $\Nn$ are what distinguish the various results. Commonly, the structure in question will be a semigroup. An early example is van der Waerden's theorem on monochromatic arithmetic progressions:

\begin{theorem}[van der Waerden]
	For every $k \in \N$ and finite colouring of $\N$, there is $a, d \in \N$ such that the arithmetic progression $a,\ a+d,\ a+2d,\ \ldots,\ a+(k-1)d$ is monochromatic.
\end{theorem}

Here, the structure in question is the semigroup $(\N,+)$. A later example, more in the style of the results of this paper, is the Hales--Jewett theorem about the word semigroup $A^{<\omega}$ over a finite alphabet $A$. We let $V = \big( A \cup \{x\} \big)^{<\omega} \setminus A^{<\omega}$ be the set of \textit{variable words}, words over $A$ which include the variable symbol $x$. Given $u \in V$ and $a \in A$, the (nonvariable) word $u[a]$ is formed by replacing each occurrence of $x$ in $u$ with $a$.

\begin{theorem}[Hales--Jewett]
	For every finite colouring of $A^{<\omega}$, there is a variable word $u \in V$ such that $\{ u[a] : a \in A \}$ is monochromatic.
\end{theorem}

Infinitary Ramsey theory received a boost in the 1970s with the advent of ultrafilter methods, as pioneered by Glazer in his proof of Hindman's finite sums theorem \cite[Thm 10.3]{comfortUltrafiltersOldNew1977}. Given a semigroup $(S,+)$, we can naturally extend $+$ to an operation $\oplus$ on the set $\beta S$ of ultrafilters on $S$. Furthermore, $\beta S$ admits a natural topology, making it a compact right-topological semigroup. The rich algebraic structure of $\beta S$ has powerful applications and consequences all throughout combinatorics \cite{hindmanAlgebraStoneCechCompactification2012}. More recently, nonstandard methods have also seen success in Ramsey theory \cite{dinassoNonstandardMethodsRamsey2019}, particularly in studying partition regularity of Diophantine equations \cite{dinassoRamseyPropertiesNonlinear2018, dinassoFermatlikeEquationsThat2018, barrettRadoConditionsNonlinear2019}.

In \cite{farahPartitionTheoremsLayered2002}, Farah, Hindman and McLeod introduced \textit{layered semigroups}, as well as shifts and layered actions thereon. The motivation was to generalise partition results about certain spaces of variable words, such as Gowers' $\FIN_k$ theorem, the Hales--Jewett theorem, and Bergelson, Blass and Hindman's theorem on located words. Layered semigroups were further explored by Lupini \cite{lupiniGowersRamseyTheorem2017} and Farmaki--Negrepontis \cite{farmakiRamseyTheoryMixed}.

This paper should be considered a ``spiritual successor'' to \cite{farahPartitionTheoremsLayered2002}. We also work in the setting of layered semigroups, but we consider a different, much broader class of morphisms, called \textit{regressive maps}. Working in this setting, we develop a general framework to prove partition theorems about a layered semigroup $S$, assuming only the existence of certain ``coherent'' sequences in $S$. This framework allows a general way to formulate and prove many fundamental results of Ramsey theory.

While \cite{farahPartitionTheoremsLayered2002} was phrased in the language of ultrafilters, we instead formalise our results using nonstandard analysis. We believe the nonstandard formulation is more intuitive, but our work has an equivalent translation in the setting of ultrafilters. \S\ref{sec:nonst} reviews the necessary concepts of nonstandard analysis, working in an internal superstructure model. Effectively, every object $M$ under consideration is assigned a nonstandard extension $\hM$, such that the \textit{transfer principle} holds---$M$ and $\hM$ satisfy the same ``elementary'' properties.

Ellis' theory of compact semitopological semigroups (CSTSs) is also essential to our nonstandard study of Ramsey theory, and we discuss the topological prerequisites in \S\ref{sec:u-semi}. For a semigroup $S$, we define a topology on $\hS$ such that $\hS$ is ``nearly'' a CSTS. This gives us analogues of results in CSTS theory---particularly the Ellis-Numakura lemma guaranteeing the existence of idempotents, which are essential to our work.

In \S\ref{sec:lay-sg}, we define layered semigroups $S$---those which can be partitioned into countably many layers $S_0, S_1, \ldots$ so that $S_0 \cup \cdots \cup S_n$ forms a semigroup, of which $S_n$ is an ideal. We see some examples which naturally occur in Ramsey theory, some of which are in fact partial, but \textit{adequate} in a specified sense. \S\ref{sec:regr} considers regressive maps on $S$---semigroup homomorphisms $f: S \to S$ which map layers downwards, and don't separate or reorder them. Natural examples of maps on layered semigroups are generally regressive, hence this notion distills the essential Ramsey-theoretic properties of such maps.

In \S\ref{sec:pi} and \S\ref{sec:seq}, nonstandard analysis comes in, as we consider sequences $(\alpha_i)$ of nonstandard elements where $\alpha_n \in \hS_n$. Such a sequence is \textit{coherent} if it is closed under all regressive maps under consideration, and \textit{Ramsey} if $\alpha_n$ absorbs all $\alpha_i$, $i \leq n$ under the semigroup operation. We present the main mechanism for proving Ramsey statements in \S\ref{sec:main}. If $\F$ is a collection of regressive maps on $S$, the framework is summarised diagrammatically below:\vspace{2mm}

\begin{tikzcd}[shorten=1mm]
%	{\parbox{layered semigroup $S$ \\ regressive maps $\F$}}
%	A \arrow[r, "\phi"] & \F\text{-coherent} \arrow[rr, "\text{Thm \ref{thm:coh->ram}}"] & & \F\text{-Ramsey} \arrow[rr, "\text{Thm \ref{thm:ram1}}"] & & \text{\parbox{0.5\textwidth}{\centering Ramsey statement\\about $(S,\F)$}} \\
	(S,\F) \arrow[r,dashed] \arrow[r,dashed,bend left,start anchor={north east},end anchor={north west}] \arrow[r,dashed,bend right,start anchor={south east},end anchor={south west}] & \F\text{-coherent} \arrow[rr, "\text{Thm \ref{thm:coh->ram}}"] & & \F\text{-Ramsey} \arrow[rr,end anchor={[xshift=2mm]},"\text{Thm \ref{thm:ram1}}"] & & \text{\begin{tabular}{c}Ramsey statement\\about $(S,\F)$\end{tabular}} \\[-4mm]
\end{tikzcd}

It is difficult to construct general arguments giving the implication \begin{tikzcd}[cramped,sep=scriptsize]\arrow[r,dashed]&{}\end{tikzcd} above, without having to impose very strong conditions on $S$ and $\F$. Therefore, the construction of an $\F$-coherent will usually depend on the specific semigroup under consideration. The most general construction we give is Lemma \ref{lem:compl}, for ``complete'' subsemigroups of $\FIN^A$---this covers Gowers' theorem and Bergelson--Blass--Hindman's theorem on located variable words.

However, the other two implications work much more generally, only requiring weak, natural conditions on $S$ and $\F$. As a result, Ramsey statements on $(S,\F)$ can be reduced to the existence of coherent sequences in $S$. Towards the end of the paper, we show how our general framework recovers many fundamental results in Ramsey theory, including
\begin{itemize}
	\item Gowers' $\FIN_k$ theorem, and its generalisation due to Lupini (\S\ref{sec:gow});
	\item The Graham--Rothschild parameter sets theorem (\S\ref{sec:gr});
	\item The Galvin--Glazer theorem and Hindman's finite sums theorem (\S\ref{sec:gg});
	\item An infinitary, multivariable generalisation of Bergelson, Blass and Hindman's partition theorem on located variable words (\S\ref{sec:lv-words}).
\end{itemize}

Collectively, these theorems imply a variety of other Ramsey-type results, including Hindman's finite unions theorem, the Hales--Jewett theorem, and van der Waerden's theorem. In each case, we give elementary nonstandard constructions of coherent sequences, which is enough to imply the corresponding result via our framework. We also present a common generalisation of Gowers' theorem and the multivariable Bergelson--Blass--Hindman theorem (and even the Milliken--Taylor theorem) in \S\ref{sec:gowers-bbh}, which is provable using our framework. Again, an elementary nonstandard argument constructs an $\F$-coherent in this case.

Throughout, we let $\N = \{ 0, 1, 2, \ldots \}$ be the set of \textit{nonnegative} integers. We may use interval notation, e.g.\ $[0,3]$, $[1,7)$, and this should be interpreted \textit{in the natural numbers}, i.e.\ $[n,m] = \{ k \in \N: n \leq k \leq m \}$. This notation will later extend to nonstandard integers $\h \N$, i.e.\ $[\xi,\zeta] = \{ \alpha \in \h\N: \xi \leq \alpha \leq \zeta \}$.

In general, we will use uppercase Latin letters $A,B,\ldots,S,T,\ldots$ for sets and semigroups, lowercase Latin letters $s,t,\ldots$ for elements thereof, and lowercase Greek letters $\alpha,\beta,\ldots$ for elements of nonstandard extensions $\hS$ of semigroups.

\section{Prerequisites}

\subsection{Nonstandard analysis}
\label{sec:nonst}

The main results of this paper (in \S \ref{sec:main}) will be proved using the tools of nonstandard analysis. Here, we give a basic overview of the concepts needed---for a more in-depth exposition of nonstandard methods and their applications to Ramsey theory, see \cite{dinassoNonstandardMethodsRamsey2019}. All our work can alternatively be formulated using ultrafilter methods, as per \cite{todorcevicIntroductionRamseySpaces2010, hindmanAlgebraStoneCechCompactification2012}.

Effectively, we work inside a ``universe'' $\V$ which includes the semigroup $(S,+)$ under consideration, as well as subsets thereof and functions $f: S^k \to S^m$. This universe comes equipped with a \textit{star map} $*: \V \to \V$, which assigns every object $M \in \V$ to its \textit{nonstandard extension} $\hM$, in such a way that the following properties hold:

\begin{axioms}[Basic properties of the star map]\label{axs:basic}\
	\begin{enumerate}
		\item For a set $A$, $\hA$ is also a set, and ${}^\sigma\hspace{-0.25em} A \subseteq \hA$, where ${}^\sigma\hspace{-0.25em} A \defeq \{ \h x: x \in A \}$. This containment is strict iff $A$ is infinite;\label{ax:ext-set}
		\item If $A, B$ are sets such that $A \subseteq B$, then $\hA \subseteq \hB$;
		\item For a set $A$ and $k \in \N$, we have $\h (A^k) = (\hA)^k$;
		\item If $f$ is a function $A \to B$, then $\hf$ is a function $\hA \to \hB$;
		\item If $f: A \to B$ and $x \in A$, then $(\hf)(x) = f(x)$;\label{ax:ext-func}
		\item If $f: A \to B$ and $x \in A$, then $(\hf)(\h\hspace{-0.1em} x) = \h (f(x))$;\label{ax:f(*x)}
		\item If $s \in S$, then $\h\hspace{-0.1em} s = s$;\label{ax:*s}
		\item If $n \in \N$, then $\h n = n$.\label{ax:*n}
	\end{enumerate}
\end{axioms}

The star map also satisfies the following key principle:

\begin{axiom}[Transfer principle]\label{ax:transfer}
	For any elementary\footnote{A formula is \textit{elementary} if all the quantifiers are bounded, i.e.\ of the form $Q\, x \in y$ for objects $x,y \in \V$. All logical formulae which we consider will be elementary.} formula $\varphi(x_1,\ldots,x_n)$ and objects $M_1,\ldots,M_n \in \V$, we have
	\begin{equation*}
		\varphi(M_1,\ldots,M_n) \text{ holds} \iff \varphi(\hM_1,\ldots,\hM_n) \text{ holds}
	\end{equation*}
\end{axiom}

For a rigorous construction satisfying Axioms \ref{axs:basic} and \ref{ax:transfer}, see \cite{dinassoNonstandardMethodsRamsey2019}. Here, we will take on faith that such a structure does exist.

The semigroup operation $+$ can be considered as a function $+: S^2 \to S$, thus we get a natural extension $\h\hspace{-0.1em} +$ of this operation to $\hS$. Abusing notation, we will use $+$ to denote both the original operation and its nonstandard extension---this is somewhat justified by Axiom \ref{axs:basic}.\ref{ax:ext-func}. We will generally do the same for functions $f: S^k \to S^m$.

A peculiarity of our approach will be that we may \textit{iterate} the star map, to obtain nonstandard extensions of nonstandard extensions, and so on. In this way, we get objects $M, \hM, \h\hM, \ldots$. We will use $\kh{n}\hspace{-0.15em} M$ to denote the $n$-fold nonstandard extension of an object $M \in \V$. Axioms \ref{axs:basic} and \ref{ax:transfer} also hold when the objects under consideration are themselves nonstandard.

\begin{remark}
	In general, the simplifying assumptions made in Axioms \ref{axs:basic}.\ref{ax:*s} and \ref{axs:basic}.\ref{ax:*n} cannot be extended to elements of $\hS$, $\h \N$ or higher in the nonstandard hierarchy. As an example, $\N$ is an initial segment of $\h\N$ \cite[Prop 2.27]{dinassoNonstandardMethodsRamsey2019}, so by transfer, $\h\N$ is an initial segment of $\h\h\N$. Now, if we take $\xi \in \h\N \setminus \N$, we have  $\h\xi \in \h\h\N \setminus \h\N$ by transfer. It follows that $\xi < \h\xi \implies \xi \neq \h\xi$.
\end{remark}

\subsection{\texorpdfstring{$u$}{u}-semigroups}
\label{sec:u-semi}

Here, we develop some further notions that prove essential in the study of Ramsey semigroups. These are mostly based on the theory of compact semitopological semigroups (see \cite[\S2]{todorcevicIntroductionRamseySpaces2010}), as developed by Ellis and others. An example is given by $\beta S$, the set of ultrafilters on $S$, whose topological and algebraic structure is well-studied \cite{hindmanAlgebraStoneCechCompactification2012, todorcevicIntroductionRamseySpaces2010}. $\beta S$ is also homeomorphic to the \textit{Stone--\v Cech compactification} of $S$ with the discrete topology. The following map allows us to transport this structure to $\hS$.

\begin{definition}\label{defn:u-eq}
	Elements $\alpha \in \hS$ generate ultrafilters on $S$ via the \textit{ultrafilter map}:
	\begin{equation*}
		\alpha \mapsto \U_\alpha = \{ A \subseteq S : \alpha \in \hA \}
	\end{equation*}
	Two elements $\alpha, \beta$ are \textit{$u$-equivalent} (denoted $\alpha \sim \beta$) if $\U_\alpha = \U_\beta$.
\end{definition}

\begin{proposition}
	$\sim$ is an equivalence relation on $\hS$.
\end{proposition}

The relation $\sim$ on $\hS$ was first considered by Di Nasso in \cite{dinassoHypernaturalNumbersUltrafilters2015}, and has seen extensive combinatorial and Ramsey-theoretic applications in \cite{dinassoNonstandardMethodsRamsey2019, dinassoRamseyPropertiesNonlinear2018, dinassoFermatlikeEquationsThat2018, barrettRadoConditionsNonlinear2019}. Some of the key properties are summarised below.

\begin{proposition}[\cite{dinassoHypernaturalNumbersUltrafilters2015,dinassoNonstandardMethodsRamsey2019}]\label{prop:dinasso}\
	\begin{enumerate}[label=(\roman*)]
		\item If $\alpha \in \hS$, $s \in S$, then $\alpha \sim s$ if and only if $\alpha = s$.\label{prop:u-eq-injec}
		\item For any function $f: S \to S$, if $\alpha \sim \beta \in \hS$ then $f(\alpha) \sim f(\beta)$.\label{prop:f-alpha}
		\item For any function $f: S \to S$, if $\alpha \in \hS$ is such that $f(\alpha) \sim \alpha$, then $f(\alpha) = \alpha$.
		\item For any $\alpha, \alpha', \beta, \beta' \in \hS$, if $\alpha \sim \alpha'$ and $\beta \sim \beta'$, then $\alpha + \hbeta \sim \alpha' + \hbeta'$.
		\item For any $\alpha \in \hS$, $\alpha \sim \halpha$.
	\end{enumerate}
\end{proposition}

There is a natural way to define a topology on $\hS$ as follows:

\begin{definition}
	For any semigroup $S$, equip $\hS$ with the \textit{$u$-topology}---that generated by the basic open sets $\hA$ for $A \subseteq S$. We say $\hS$ is a \textit{compact $u$-semigroup}, i.e.
	\begin{enumerate}
		\item $\hS$ is compact;
		\item For any $\alpha, \beta \in S$, there exists $\gamma \in S$ such that $\gamma\, \sim\, \alpha + \hbeta$;
		\item The map $\alpha\, \mapsto\, \alpha + \hbeta$ is continuous.
	\end{enumerate}
\end{definition}

\begin{proposition}
	$\hS / {\sim}$ is Hausdorff, i.e. two elements $\alpha, \beta \in \hS$ are inseparable by disjoint open sets exactly when $\alpha \sim \beta$.
\end{proposition}

\begin{corollary}\label{zero}
	For continuous functions $f,g: \hS \to \hS$, the set $\{ \alpha \in \hS : f(\alpha) \sim g(\alpha) \}$ is closed.
\end{corollary}

\begin{proposition}
	For any function $f: S \to S$, its nonstandard extension $\hf: \hS \to \hS$ is continuous with respect to the $u$-topology on $\hS$.
\end{proposition}

Idempotent ultrafilters are key to most applications of infinitary methods in Ramsey theory---their existence follows from the Ellis--Numakura lemma. Throughout this paper, we will use a similar notion of idempotence for elements of $\hS$.

\begin{definition}
	Suppose $(S,+)$ is a semigroup. We say $\alpha \in \hS$ is \textit{$u$-idempotent} if $\alpha + \halpha\ \sim\ \alpha$.
\end{definition}

\begin{lemma}[Ellis--Numakura]\label{lem:u-id}
	If $T \subseteq \hS$ is a closed $u$-subsemigroup, then $T$ contains a $u$-idempotent element.
\end{lemma}

Often, we will need a strengthening of Lemma \ref{lem:u-id}, as follows.

\begin{definition}
	Define a relation $\preccurlyeq$ on $\hS$ by
	\begin{equation*}
		\alpha \preccurlyeq \beta\ \iff\ \alpha + \hbeta\ \sim\ \beta + \halpha\ \sim\ \alpha
	\end{equation*}
	$\preccurlyeq$ is a partial order (up to $u$-equivalence) on the $u$-idempotents of $\hS$.
\end{definition}

\begin{corollary}[{\cite[Lemma 2.3]{todorcevicIntroductionRamseySpaces2010}}]\label{cor:u-id-min}
	Any closed $u$-subsemigroup of $\hS$ contains a $\preccurlyeq$-minimal $u$-idempotent element.
\end{corollary}

\section{Layered semigroups}
\label{sec:lay-sg}

Our results concern the framework of \textit{layered semigroups}, as introduced by Farah, Hindman and McLeod in \cite{farahPartitionTheoremsLayered2002}. We will work with the following adaptation of their definition:

\begin{definition}\label{defn:lay-sg}
	A \textit{layered semigroup} is a (total) semigroup $S$, with a \textit{layering map} $\ell: S \to \N$ such that for all $s,t \in S$, $\ell(s+t) = \max \{ \ell(s), \ell(t) \}$.
\end{definition}

The map $\ell$ splits $S$ into layers $S_n = \ell^{-1}(n)$ for each $n \in \range(\ell)$. Without loss of generality, we will suppose\footnote{By shifting down values of $\ell$ as required.} that $\range(\ell)$ is an initial segment of $\N$. We will encounter situations where $\range(\ell)$ is finite (i.e.\ our semigroup has finitely many layers), but also cases when $\range(\ell) = \N$ (i.e.\ our semigroup has infinitely many layers. In theory, we could allow $\range(\ell) = \delta$ for ordinals $\delta > \omega$, but we will not pursue such generalisations here.

Definition \ref{defn:lay-sg} is equivalent to the following, which is more in the style of Farah, Hindman and McLeod's original definition:

\begin{proposition}
	A pair $\big( S,\ \ell: S \to \N \big)$ form a layered semigroup if and only if:
	\begin{enumerate}
		\item $S_{\leq n} \defeq \ell^{-1} \big( \{ 0, \ldots, n \} \big)$ is a subsemigroup of $S$;
		\item $S_n = \ell^{-1}(n)$ is a two-sided ideal of $S_{\leq n}$.
	\end{enumerate}
\end{proposition}

%\begin{definition}
%	An \textit{$\omega$-layered semigroup} is a (total) semigroup $S$, with a distinguished ordered partition $(S_i)_{i=0}^\infty$ of $S$, such that for all $i \in \N$:
%	%nested sequence $S_0 \subseteq S_1 \subseteq S_2 \subseteq \cdots$ of nonempty (total) semigroups, such that 
%	\begin{enumerate}
%		\item $T_i \defeq \bigcup_{j=0}^i S_j$ is a subsemigroup of $S$;
%		\item $S_i$ is a nonempty two-sided ideal of $T_i$.
%	\end{enumerate}
%	The $S_i$ are called the \textit{layers} of $S$, with $S_k$ being the $k$th layer.
%	
%	{\colr{red} Alternatively, a semigroup $S$ with a surjective map $\ell: S \to \N$ such that for all $s,t \in S$, $\ell(st) = \max \{ \ell(s), \ell(t) \}$. Equivalent under $s \in S_i \iff \ell(s) = i$. This definition possibly more elegant}
%\end{definition}

\begin{remark}
	Farah, Hindman and McLeod's original definition in \cite{farahPartitionTheoremsLayered2002} is given by only allowing finitely many layers, and further positing that $S_0 = \{ e \}$, where $e$ is a two-sided identity for any element of $S$. We have relaxed both conditions, since we will encounter layered semigroups for which neither holds.
\end{remark}

To illustrate Definition \ref{defn:lay-sg}, we present some examples of layered semigroups which naturally arise in Ramsey theory, and which will be relevant later in this paper.

\begin{example}\label{inf1}
	For each $i$, let $M_i$ be a monoid (semigroup with identity $e_i$), such that only the identity has an inverse. Let $S$ be the set of tuples $(m_0,m_1,m_2,\ldots) \in \prod_{i=0}^\infty M_i$ with finite support (i.e. $m_n \neq e_n$ for finitely many $n$). Then, $S$ is a layered semigroup under pointwise operations, and the layering map $\ell \big[ (m_i) \big] = \min \{ k:\, \forall i \geq k\ m_i \neq e_i \}$.
	%where the layers are $S_0 = \{ (e_0,e_1,e_2,\ldots) \}$, and
%	\begin{equation*}
%	S_i \defeq \{ f \in K :\ f(i-1) \neq e,\ f(j) = e \text{ for all } j \geq i \}
%	\end{equation*}
%	for all $i>0$.
\end{example}

\begin{example}\label{exm:gr-sg}
	Fix a finite alphabet $A$. For $k \in \N$, a \textit{$k$-parameter word} is an element of $(A \cup \{ x_1, x_2, \ldots, x_k \})^{<\omega}$ such that all the variables $x_1, x_2, \ldots, x_k$ appear, and their \textit{first} appearances are in increasing order. Let $W_k$ be the set of all $k$-parameter words.\footnote{So $W_0$ is simply $A^{<\omega}$, the set of all words over $A$.} Then, $W = \bigcup_{k=0}^\infty W_k$ is a layered semigroup\footnote{The layering map $\ell: W \to \N$ here maps every parameter word to the number of variables it contains.} under concatenation, called the \textit{Graham--Rothschild semigroup.}
\end{example}

%\begin{example}
%	Let $\FIN$ be the set of all functions $f: \N \to \N$ with finite support, i.e.\ $\supp(f) \defeq \{ n \in \N : f(n) \neq 0 \}$ is finite. We define a partial operation $+$ on $\FIN$ by $f+g$ iff $f(n) \times g(n) = 0$ for all $n \in \N$.\footnote{Equivalently, if $f$ and $g$ have disjoint supports.} Then, $\FIN$ is a partial adequate layered semigroup under the map $\ell(f) = \max \big( \range(f) \big)$.
%\end{example}

We will also consider \textit{partial} semigroups---those for which the operation is not always defined. To exclude trivial cases, such as when the operation is \textit{never} defined, we have the following notion of \textit{adequacy} for a partial semigroup.

\begin{definition}
	A partial semigroup $S$ is \textit{adequate} \cite[Defn 1.15.6]{farahPartitionTheoremsLayered2002, hindmanAlgebraStoneCechCompactification2012} or \textit{directed} \cite{todorcevicIntroductionRamseySpaces2010} if, for any finite subset $F \subseteq S$, there exists $y$ such that $x+y$ is defined for all $x \in F$.
\end{definition}

We now generalise Definition \ref{defn:lay-sg} to the case of partial semigroups. However, it is not enough for just $S$ to be adequate---we need each layer to be adequate also.

\begin{definition}
	An \textit{adequate partial layered semigroup} is a partial semigroup $S$, with a \textit{layering map} $\ell: S \to \N$ such that:
	\begin{enumerate}
		\item For all $s,t \in S$ such that $s+t$ is defined, $\ell(s+t) = \max \{ \ell(s), \ell(t) \}$;
		\item For all $s_1,\ldots,s_n \in S$ with $\ell(s_1) = \cdots = \ell(s_n)$, there is $t \in S$ such that $\ell(t) = \ell(s_i)$ and $s_i + t$ is defined for all $i \leq n$.
	\end{enumerate}
\end{definition}

\begin{example}\label{exm:gow-sg}
	Let $\FIN$ be the set of all functions $f: \N \to \N$ with finite support, i.e.\ $\supp(f) \defeq \{ n \in \N : f(n) \neq 0 \}$ is finite. For $f, g \in \FIN$, we define $f + g$ pointwise iff the pointwise product $f(n) g(n) = 0$ for all $n \in \N$.\footnote{Equivalently, if $f$ and $g$ have disjoint supports.} Then, $\FIN$ is an adequate partial layered semigroup under the layering map $\ell(f) = \max(\range(f))$, called the \textit{Gowers semigroup.}
	%$\FIN_{<\omega} = \bigcup_{k=0}^\infty \FIN_k$ is an adequate $\omega$-layered semigroup, called the \textit{Gowers semigroup.} By Remark $\ref{adeq}$, our discussion will apply to $\FIN_{<\omega}$.
\end{example}

\begin{example}
	Later, we will consider nonstandard extensions of these semigroups. We use transfer to deduce what these extensions look like. For example, $\h \FIN_k$ will consist\footnote{Not all such functions are in $\h \FIN_k$---only those that are \textit{internal}. \cite[\S 2.5]{dinassoNonstandardMethodsRamsey2019} gives a good overview of internal/external objects. An understanding of these will not be necessary in this paper.} of functions $\varphi: \h \N \to [0,k]$ with \textit{hyperfinite support} (i.e. $\supp(\varphi) \subseteq [0,\xi]$ for some $\xi \in \h \N$) and having $\max(\range(\varphi)) = k$. This is clear from writing the definition of $\FIN_k$ as an elementary formula, and applying transfer.
\end{example}

\begin{remark}
	In general, $\bigcup_{i<\delta} \hS_i \subseteq \hS$, but this containment may be strict when $S$ has infinitely many layers. For example, in the Graham--Rothschild semigroup $W$ of Example \ref{exm:gr-sg}, $\h W$ consists of parameter words having \textit{hyperfinite} length\footnote{This is clear by considering a $k$-parameter word $w \in W$ as a function $w: [1,n] \to A \cup \{ x_1, x_2, \ldots, x_k \}$ for some $n \in \N$, and applying transfer.} $\xi$ and $\zeta$-many variables for some $\xi,\zeta \in \h\N$. In contrast, $\bigcup_{i=0}^\infty \h W_i$ is the subset of $\h W$ consisting of words with only \textit{finitely} many variables. We will never need to consider \textit{all} of $\hS$; just the $\hS_i$ will be enough for our purposes.
\end{remark}

\subsection{The \texorpdfstring{$*$}{*}-product\texorpdfstring{ $\Pi_S$}{}}
\label{sec:pi}

We will consider sequences $(\alpha_i)$ of nonstandard elements, where each $\ell(\alpha_i) = i$. For notational convenience, we define the following:

\begin{definition}\label{defn:pi}
	Given a \textit{total} layered semigroup $S$, define
	\begin{equation*}
	\Pi_S \defeq \prod_{i < \delta} \hS_i
	\end{equation*}
	where $\delta$ is the number of layers.	If it is clear what layered semigroup we are referring to, we may just use the notation $\Pi$.
\end{definition}

We now generalise Definition \ref{defn:pi} to adequate partial semigroups.

\begin{definition}
	For a subsemigroup $A \subseteq S$, we define
	\begin{equation*}
		\gA \defeq \{ \alpha \in \hA: x + \alpha \text{ is defined for all } x \in A \}
	\end{equation*}
\end{definition}

\begin{example}
	Let $\FIN$ be the Gowers semigroup of Example \ref{exm:gow-sg}. Then, for $k \in \N$, $\g\FIN_k$ consists of the \textit{cofinite} functions $\varphi \in \h\FIN_k$---those whose supports are disjoint from $\N$.
\end{example}

\begin{remark}
	In the case of partial adequate semigroups $S$, generally the whole of $\beta S$ is not considered, but only a special subset, which is notated $\gamma S$ \cite{todorcevicIntroductionRamseySpaces2010} or $\delta S$ \cite{hindmanAlgebraStoneCechCompactification2012, farahPartitionTheoremsLayered2002}. For each $A \subseteq S$, $\gA$ is the preimage of $\gamma A \subseteq \beta A$ under the ultrafilter map (Definition \ref{defn:u-eq}), hence the notation.
\end{remark}

The sets $\gA$ are useful because they have the following property:

\begin{proposition}
	If $\alpha \in \hA$ and $\beta \in \gA$, then $\alpha + \hbeta$ is \textit{always} defined.
\end{proposition}
	
\begin{proof}
	Since $\beta \in \gA$, the elementary formula $$\varphi(A,\hA,\beta)\ =\ \forall x \in A\ \ \exists y \in \hA\ \ x+\beta=y$$ holds, so by transfer, $$\varphi(\hA,\h\hA,\hbeta)\ =\ \forall x \in \hA\ \ \exists y \in \h\hA\ \ x+\hbeta=y$$ also holds. Letting $x = \alpha$ gives the result.
\end{proof}

\begin{proposition}
	If $A \subseteq S$ is adequate, then $\gA$ is nonempty.
\end{proposition}

\begin{proof}
	For each $s \in A$, let $K_s = \{ y \in A: s+y \text{ is defined} \}$. Then we have
	\begin{equation*}
		\gA = \bigcap_{s \in A} \hK_s
	\end{equation*}
	By adequacy, each $K_s$ is nonempty, so each $\hK_s$ is nonempty, and closed in the $u$-topology on $A$. The collection $\{ \hK_s: s \in A \}$ has the finite intersection property, since for any finite $F \subseteq A$:
	\begin{equation*}
		\bigcap_{s \in F} \hK_s\ =\ \h\hspace{-0.3em} \left[\, \bigcap_{s \in F} K_s \right]
	\end{equation*}
	which is nonempty by adequacy of $A$. The result follows by compactness of $\hA$ (it is a closed subset of the compact space $\hS$).
\end{proof}

\begin{definition}\label{defn:pi2}
	Given an adequate layered semigroup $S$, define
	\begin{equation*}
		\Pi_S \defeq \prod_{i < \delta} \gS_i
	\end{equation*}
	where $\delta$ is the number of layers.	If it is clear what layered semigroup we are referring to, we may just use the notation $\Pi$.
\end{definition}

Notice that if $S$ is total, then $\gA = \hA$ for any $A \subseteq S$, so Definitions \ref{defn:pi} and \ref{defn:pi2} coincide.

%\begin{notn}
%	Given a Gowers system $(S_i)_{i=1}^\infty$, we let $S = \bigcup_{i=1}^\infty S_i$. By convention, we will set $S_j = \varnothing$ for all $j \notin \{ 1, 2, 3, \ldots \}$.
%\end{notn}

%Here, we will assume that semitopological means right-topological, i.e. right multiplication $x \mapsto xs$ is continuous for all $s \in S_i$. All our theorems are also valid for left-topological semigroups, simply by reversing the operands.

\begin{proposition}\label{prop:pi-compact}
	For any layered semigroup $S$, $\Pi_S$ is a compact $u$-semigroup, with the product topology and operation defined componentwise.
\end{proposition}

\section{Regressive maps}
\label{sec:regr}

Our Ramsey-type results will be in the context of certain functions acting on layered semigroups. We will require our functions $f: S \to S$ to have the following properties.

\begin{definition}\label{defn:regr}
	A \textit{regressive map} is a function $f: S \to S$ such that for all $s, t \in S$:%It is natural to require that all functions under consideration are semigroup homomorphisms. The "layering" of our semigroups also makes the following definition natural.
	\begin{enumerate}
		\item $f(st) = f(s)f(t)$, i.e.\ $f$ is a semigroup homomorphism;
		\item $\ell(f(s)) \leq \ell(s)$;
		\item $\ell(s) \leq \ell(t) \implies \ell(f(s)) \leq \ell(f(t))$;
		\item $\abs{\ell(f(s)) - \ell(f(t))}\ \leq\ \abs{\ell(s) - \ell(t)}$;
	\end{enumerate}
\end{definition}

\begin{remark}
	Property (i) seems natural, since we are dealing with semigroups. The ``layering'' of our semigroups motivates property (ii). More specifically, in \S\ref{sec:main}, we will inductively construct ``coherent'' sequences $\alpha^{(0)} \in \gS_0,\ \alpha^{(1)} \in \gS_1,\ \ldots$ which are closed under all the functions $f$ under consideration. Property (ii) makes such a construction tractable, since we can build each $\alpha^{(k)}$ based only on the $\alpha^{(j)}$ for $j < k$.
	
	Properties (iii) and (iv) are harder to motivate, but they are essential to the proof of Lemma \ref{lem:seq-sum}, which plays a big part in the proof of Theorem \ref{thm:coh->ram}. We will see below that many natural examples of maps on layered semigroups satisfy these two properties.
\end{remark}

We will consider sequences which are ``well-behaved'' with respect to a \textit{collection} $\F$ of regressive maps on $S$. Generally, $\F$ will be closed under composition---however, this is not required for our arguments to work. To make effective use of the transfer principle, we do require the following:

\begin{definition}
Let $\F$ be a collection of regressive maps $f: S \to S$. $\F$ is \textit{locally finite} (l.f.) if for all $i \in \N$, the set $\F_i = \{ f\vert_{S_{\leq i}} : f \in \F \}$ is finite.
\end{definition}

\begin{example}\label{exm:gr-maps}
	Let $W$ be the Graham--Rothschild semigroup of Example \ref{exm:gr-sg}. Now, we consider infinite parameter words $\tilde{w}$ - elements of $(A \cup \{ x_1, x_2, \ldots \})^\omega$ such that all $x_i$ appear, with their initial appearances in increasing order.
	
	Every such $\tilde{w}$ defines a function $W \to W$, $u \mapsto u[\tilde{w}]$, called the \textit{substitution map}, by replacing each occurrence of $x_i$ in $u$ with the $i$th character in $\tilde{w}$. Each substitution map is a regressive map, and the collection $\F$ of all such maps is locally finite, and closed under composition.
\end{example}

\begin{example}\label{exm:gow-maps}
	Let $\FIN$ be the Gowers semigroup of Example \ref{exm:gow-sg}. Every $F: \N \to \N$ induces a map $\tilde{F}: \FIN \to \FIN$ by composition, i.e. $\tilde{F}(f) = F \circ f$. When $F$ is a nondecreasing surjection, $\tilde{F}$ is a regressive map, and the collection $\F$ of all such maps is locally finite, and closed under composition. These maps, first considered in \cite{bartosovaGowersRamseyTheorem2017}, are called \textit{(generalised) tetris operations.}
\end{example}

\subsection{Special sequences\texorpdfstring{ in $\Pi_S$}{}}
\label{sec:seq}

Throughout this section, fix a layered semigroup $S$ and a locally finite collection of regressive maps $\F$ on $S$. Essential to the proof of Theorem \ref{thm:coh->ram} is the following notion of \textit{coherence} for elements of $\Pi_S$.

\begin{definition}\label{defn:coh}
	An element $(\alpha_i)_{i < \delta} \in \Pi_S$ is \textit{$\F$-coherent} if for all $f \in \F$ and $j < \delta$, we have $f(\alpha_j) \sim \alpha_k$ for some\footnote{Such a $k$ is uniquely defined.} $k \leq j$.
\end{definition}

Intuitively, $(\alpha_i)$ is $\F$-coherent if it is closed under all functions $f \in \F$. Definition \ref{defn:coh} doesn't say anything about the product of elements in $(\alpha_i)$, so a stronger notion of coherence is needed to draw conclusions about products.

\begin{definition}\label{defn:f-ram}
	An $\F$-coherent element $(\alpha_i)_{i < \delta} \in \Pi_S$ is \textit{$\F$-Ramsey} if, for all $i \leq j < \delta$, $\alpha_i \succcurlyeq \alpha_j$.\footnote{For the case $i=j$, this implies $\alpha_i$ is $u$-idempotent.} We say $S$ itself is \textit{$\F$-Ramsey} if $\Pi_S$ contains an $\F$-Ramsey element. 
\end{definition}

\begin{remark}
	Sequences of ultrafilters satisfying the analogue of Definition \ref{defn:f-ram} (as well as being $\preccurlyeq$-minimal) are termed \textit{reductive} in \cite{hindmanCombiningExtensionsHalesJewett2019}.
\end{remark}

%Here, we define some further conditions on the regressive maps $f: S \to S$, which prove necessary for many of our results.
%
%\begin{definition}
%	For each $s \in S$, let $\ell(s)$ be the unique $i$ such that $s \in S_i$. We say $f: S \to S$ is \textit{$\ell$-monotonic} if it is weakly increasing with respect to $\ell$, i.e. for all $s, t \in S$:
%	\begin{equation*}
%		\ell(s) \leq \ell(t) \implies \ell(f(s)) \leq \ell(f(t))
%	\end{equation*}
%\end{definition}
%
%\begin{definition}
%	Note that $d(s,t) = \abs{\ell(s) - \ell(t)}$ is a (pseudo)metric on $S$. A semigroup homomorphism $f: S_i \to S_i$ is \textit{metric} if it is a metric map, or 1-Lipschitz, with respect to $d$, i.e. for all $s, t \in S$:
%	\begin{equation*}
%		\abs{\ell(f(s)) - \ell(f(t))}\ \leq\ \abs{\ell(s) - \ell(t)}
%	\end{equation*} 
%\end{definition}

\section{Main results}
\label{sec:main}

Throughout this section, fix a layered semigroup $S$ and a locally finite collection of regressive maps $\F$ on $S$. The existence of an $\F$-Ramsey in $\Pi_S$ is a powerful statement---it implies that general Ramsey statements (Theorems \ref{thm:ram1} and \ref{thm:ram2}) are true of our layered semigroup $S$. When put into context, these general statements reduce to familiar results of Ramsey theory. Before proving these statements, we need the following concept:

\begin{definition}
	A sequence $(x_i) \subseteq S$ is called a \textit{block sequence} \cite{gowersLipschitzFunctionsClassical1992,lupiniGowersRamseyTheorem2017} or \textit{basic sequence} \cite[Thm 2.20]{todorcevicIntroductionRamseySpaces2010} if, for any $n_0 < \cdots < n_{\ell-1}$ and $f_0, \ldots, f_{\ell-1} \in \F$, $f_0(x_{n_0}) + \cdots + f_{\ell-1}(x_{n_{\ell-1}})$ is defined.
\end{definition}

\begin{remark}
	The above definition has content only when $S$ is partial - every sequence in a total semigroup is a block sequence.
\end{remark}

\begin{theorem}\label{thm:ram1}
	Suppose $S$ is $\F$-Ramsey with $\delta$ layers, and $n < \delta$. Then, for any finite colouring of $S$, there exists a block sequence $(x_i)_{i=1}^\infty \subseteq S_n$ such that for every $k \leq n$, the following set is monochromatic:
	\begin{equation*}
		S_k\ \cap\ \big\{\ f_1(x_{n_1}) + \cdots + f_\ell(x_{n_\ell})\ :\ 1 \leq n_1 < \cdots < n_\ell,\ f_1, \ldots, f_\ell \in \F\ \big\}
	\end{equation*}
\end{theorem}

We sketch the proof. Note that for an $\F$-Ramsey $(\alpha_i)_{i < \delta} \in \Pi$, the constant sequence $(\alpha_n, \alpha_n, \alpha_n, \ldots)$ effectively satisfies the conclusion of the theorem, except that it is in $\gS_n$ rather than $S_n$. Given the existence of $\alpha_n$, we repeatedly apply transfer to deduce the existence of each $x_i$.

\begin{proof}[Proof of Theorem \ref{thm:ram1}]
	Fix an $\F$-Ramsey sequence $(\alpha_i)_{i < \delta} \in \Pi$. For each $i < \delta$, let $A_i = \big\{ x \in S_i :\ c(x) = c(\alpha_i) \big\}$, and $A = \bigcup_{i < \delta} A_i$.
	
	$\alpha_n$ witnesses that the sentence\footnote{This is a finite sentence, since $\F$ is locally finite by assumption.}
	\begin{equation*}
		\exists\ \tau \in \hS_n:\ \bigwedge_{f_1,f_2 \in \F_n} \big( f_1(\tau) \in \hA\ \land\ f_1(\tau) + \hf_2(\alpha_n) \text{ is defined and in } \h\hA \big)
	\end{equation*}
	is true, so by transfer, there exists $x_1 \in S_n$ such that, for all $f_1, f_2 \in \F$, $f_1(x_1) \in A$, and $f_1(x_1) + f_2(\alpha_n)$ is defined and in $\hA$.
	
	We also have $f_1(x_1) + f_2(\alpha_n) + \hf_3(\alpha_n) \in \h\hA$ for any choice of $f_1, f_2, f_3 \in \F$. This is because $f_2(\alpha_n) \sim \alpha_i$ and $f_3(\alpha_n) \sim \alpha_j$ for some $i,j \leq n$, so either $f_2(\alpha_n) + \hf_3(\alpha_n) \sim \alpha_i \sim f_2(\alpha_n)$, or $f_2(\alpha_n) + \hf_3(\alpha_n) \sim \alpha_j \sim f_3(\alpha_n)$. Hence, $\alpha_n$ witnesses the truth of
	\begin{flalign*}
		\exists\ \tau \in \hS_n:\ \bigwedge_{f_1,f_2,f_3 \in \F_n} \big(\ f_1(\tau) \in \hA\ \land\ f_1(&\tau) + \hf_2(\alpha_n) \text{ is defined and in } \h\hA\ \\[-3mm]
		\land\ f_1(x_1) + f_2(&\tau) \text{ is defined and in } \hA\ \\%[-3mm]
		\land\ f_1(x_1) + f_2(&\tau) + \hf_3(\alpha_n) \text{ is defined and in } \h\hA\ \big)
	\end{flalign*}
	whence by transfer, we get $x_2 \in S_n$ with similar properties.
	
	In general, suppose we have defined $x_1, \ldots, x_{j-1} \in S_n$ such that for all $\ell < j$, all $n_1 < \cdots < n_\ell < j$, and all $f_1, \ldots, f_\ell, f \in \F$, we have
	\begin{flalign*}
		&f_1(x_{n_1}) + \cdots + f_\ell(x_{n_\ell}) \text{ is defined and in } A \\
		&f_1(x_{n_1}) + \cdots + f_\ell(x_{n_\ell}) + f(\alpha_n) \text{ is defined and in } \hA
	\end{flalign*}
	Then, by a similar argument to before, it is also true that $f_1(x_{n_1}) + \cdots + f_\ell(x_{n_\ell}) + f(\alpha_n) + \hf'(\alpha_n)$ is defined and in $\h\hA$ for any choice of $f' \in \F$. Thus, $\alpha_n$ witnesses the truth of
	\begin{flalign*}
		\exists\ \tau \in \hS_n:\ \bigwedge_{\ell < j} \quad \bigwedge_{n_1 < \cdots < n_\ell < j} \quad &\bigwedge_{f_1, \ldots, f_\ell, f, f' \in \F} \\[1mm]
		\hspace*{10mm} \big(\ f_1(x_{n_1}) + \cdots + f_\ell(x_{n_\ell}) + f(&\tau) \text{ is defined and in } \hA \\
		\hspace*{20mm} \land\ f_1(x_{n_1}) + \cdots + f_\ell(x_{n_\ell}) + f(&\tau) + \hf'(\alpha_n) \text{ is defined and in } \h\hA\ \big)
	\end{flalign*}
	so by transfer, we get $x_j \in S_n$ with similar properties. Continue ad infinitum.
\end{proof}

\begin{theorem}\label{thm:ram2}
	Suppose $S$ is $\F$-Ramsey with $\delta$ layers. Then, for any finite colouring of $S$, there exists a block sequence $(x_i)_{i < \delta} \in \prod_{i < \delta} S_i$ such that for every $k < \delta$, the following set is monochromatic:
	\begin{equation*}
		S_k\ \cap\ \big\{\ f_1(x_{n_1}) + \cdots + f_\ell(x_{n_\ell})\ :\ 0 \leq n_1 < \cdots < n_\ell < \delta,\ f_1, \ldots, f_\ell \in \F\ \big\}
	\end{equation*}
\end{theorem}

\begin{proof}
	Identical to that of Theorem \ref{thm:ram1}, but at each stage when defining $x_i$, replace $\alpha_n$ with $\alpha_i$, and $S_n$ with $S_i$. Continue up to stage $\delta$.
\end{proof}

Theorem \ref{thm:ram1} easily implies a finite version:

\begin{theorem}\label{thm:ram-fin-seq}
	Suppose $S$ is $\F$-Ramsey with $\delta$ layers, and $n < \delta$, $m \in \N$. Then, for any finite colouring of $S$, there exists a block sequence $(x_1,\ldots,x_m) \subseteq S_n$ of length $m$ such that for every $k \leq n$, the following set is monochromatic:
	\begin{equation*}
		S_k\ \cap\ \big\{\ f_1(x_{n_1}) + \cdots + f_\ell(x_{n_\ell})\ :\ n_1 < \cdots < n_\ell \leq m,\ f_1, \ldots, f_\ell \in \F\ \big\}
	\end{equation*}
\end{theorem}

\begin{corollary}[\textnormal{$m=1$}]\label{cor:ram-fin}
	Suppose $S$ is $\F$-Ramsey with $\delta$ layers, and $n < \delta$. Then, for any finite colouring of $S$, there exists $x \in S_n$ such that for every $k \leq n$, $S_k \cap \{ f(x): f \in \F \}$ is monochromatic.
\end{corollary}

Incredibly, given the existence of \textit{any} $\F$-coherent element in $\Pi_S$, we can construct an $\F$-Ramsey, and thus show that the general Ramsey statements above hold in $S$. The following lemma is essential to this construction.

\begin{lemma}\label{lem:seq-sum}
	Suppose that $(\alpha_i)_{i < \delta} \in \Pi_S$ is $\F$-coherent and $u$-idempotent. Then, for any $f \in \F$ and $k < \delta$, there is $j = j_k \leq k$ such that
	\begin{equation*}
		f \big[ \alpha_k + \halpha_{k-1} + \cdots + \khalpha{(k-1)}_1 + \khalpha{k}_0 \big]\ \sim\ \alpha_j + \halpha_{j-1} + \cdots + \khalpha{(j-1)}_1 + \khalpha{j}_0
	\end{equation*}
\end{lemma}

\begin{proof}
	By induction on $k$. Given $f \in \F$, the $\F$-coherence of $(\alpha_i)$ implies that $f(\alpha_0) \sim \alpha_0$. Thus, the base case $k=0$ holds with $j_0 = 0$. Now, suppose the lemma holds for $k>0$---we prove the $k+1$ case. By $\F$-coherence, there is $\ell \leq {k+1}$ such that $f(\alpha_{k+1}) \sim \alpha_\ell$, and Definition \ref{defn:regr}.(iii) and \ref{defn:regr}.(iv) imply that either $\ell = j_k$ or $\ell = j_k+1$.
	\begin{description}
		\item[Case 1:] $\ell = j_k$. Then,
		\begin{align*}
			&\ f \Big[ \alpha_{k+1} + \big( \halpha_k + \h\halpha_{k-1} + \cdots + \khalpha{k}_1 + \khalpha{(k+1)}_0 \big) \Big] \\
			\sim\ &\ f(\alpha_{k+1}) + f \big[ \halpha_k + \h\halpha_{k-1} + \cdots + \khalpha{k}_1 + \khalpha{(k+1)}_0 \big] &&\text{$f$ homomorphism} \\
			\sim\ &\ \alpha_j + \halpha_j + \h\halpha_{j-1} + \cdots + \khalpha{j}_1 + \khalpha{(j+1)}_0 \\
			\sim\ &\ \alpha_j + \h\hspace{-0.1em} \big( \halpha_{j-1} + \cdots + \khalpha{(j-1)}_1 + \khalpha{j}_0 \big) &&\text{$u$-idempotence of $\alpha_j$} \\
			\sim\ &\ \alpha_j + \halpha_{j-1} + \cdots + \khalpha{(j-1)}_1 + \khalpha{j}_0 &&\text{Proposition \ref{prop:dinasso}}
		\end{align*}
		
		\item[Case 2:] $\ell = j_k+1$. We can prove similarly that
	\end{description}
	\begin{equation*}
		f \big[ \alpha_{k+1} + \halpha_k + \cdots + \khalpha{k}_1 + \khalpha{(k+1)}_0 \big]\ \sim\ \alpha_{j+1} + \halpha_j + \cdots + \khalpha{j}_1 + \khalpha{(j+1)}_0 \hfill% \qedhere
	\end{equation*}
	In either case, $j_{k+1} \defeq \ell$ is as required.
\end{proof}

\begin{theorem}\label{thm:coh->ram}
	Suppose $S$ has an $\F$-coherent element $(\alpha_i)_{i<\delta} \in \Pi$. Then, $S$ is $\F$-Ramsey.
\end{theorem}

\begin{proof}
	The following argument is based on \cite{lupiniGowersRamseyTheorem2017}. Recall that $\Pi$ is a compact $u$-semigroup by Proposition \ref{prop:pi-compact}. We show that the subset $\Pi_\F \subseteq \Pi$ of $\F$-coherent elements is a closed $u$-subsemigroup.
	
	$\Pi_\F$ closed: apply Lemma \ref{zero}, writing $\Pi_\F$ as
	\begin{equation*}
		\Pi_\F\ =\ \bigcap_{j < \delta}\ \bigcup_{k \leq j}\ \big\{ (\alpha_i) \in \Pi : f(\alpha_j) \sim \alpha_k \big\}
	\end{equation*}
	
%	we show the complement is open. Take some $(\beta_i)_{i=0}^\infty \notin \Pi_\F$, and some $f \in \F$ such that there exists $k \in \N$ with $f(\beta_k) \nsim \beta_j$ for any $j \in \N$. {{red} something something disjoint neighbourhoods}
	
	$\Pi_\F$ $u$-subsemigroup: take $(\beta_i)_{i=0}^\infty, (\gamma_i)_{i=0}^\infty \in \Pi_\F$. Then, $(\beta_i) + \h (\gamma_i) = (\beta_i + \h \gamma_i)$ is $\F$-coherent, by Proposition \ref{prop:dinasso} and the fact that all $f \in \F$ are homomorphisms. By transfer, there exists $(\delta_i) \in \Pi_\F$ with $(\delta_i) \sim (\beta_i + \h \gamma_i)$, as required.
	
	It follows that $\Pi_\F$ is itself a compact $u$-semigroup. By induction on $k < \delta$, we will construct a sequence of elements $\big( \alpha^{(k)}_i \big)_{i<\delta} \in \Pi_\F$ with the following properties:
	\begin{enumerate}[label=(\roman*),itemsep=0mm]
		\item For all $i \leq k$, $\alpha^{(k)}_i \sim \alpha^{(k-1)}_i$;
%		\item The sequence $\big( \alpha^{(k)}_i \big)_{i=0}^\infty$ is $\F$-coherent;
		\item For all $i \in \N$, $\alpha^{(k)}_i$ is $u$-idempotent;
		\item For all $i \leq k$ and $j \geq i$, we have $\alpha^{(k)}_j + \halpha^{(k)}_i\ \sim\ \alpha^{(k)}_j$.
	\end{enumerate}\vspace*{2mm}
	
	To begin, take any $u$-idempotent $\big( \alpha'_i \big)_{i<\delta} \in \Pi_\F$ by Lemma \ref{lem:u-id}. Let
	\begin{equation*}
		Z_0 = \left\{ \big( \beta_i \big)_{i<\delta} \in \Pi_\F :\ \beta_0 \sim \alpha'_0,\ \forall j\ \beta_j + \halpha'_0\ \sim\, \beta_j \right\}
	\end{equation*}
	Then, $Z_0$ is a compact $u$-semigroup by Lemma \ref{zero}. Furthermore, it is nonempty, since it contains the element $\big( \beta_i \big)_{i<\delta}$ defined by
	\begin{equation*}
		\beta_i\ \sim\ \alpha'_i + \halpha'_0
	\end{equation*}
	
	To see this is in $\Pi_\F$, observe that for any $f \in \F$, $f \big[ \alpha'_0 \big] \sim \alpha'_0$, since $f$ is regressive and $\big( \alpha'_i \big)$ is $\F$-coherent. Then we have
	\begin{equation*}
		f(\beta_i)\ \sim\ f \big[ \alpha'_i + \halpha'_0 \big]\ \sim\ f \big[ \alpha'_i \big] + \hf \big[ \alpha'_0 \big]\ \sim\ \alpha'_j + \halpha'_0\ \sim\ \beta_j
	\end{equation*}
	for some $j \leq i$. It follows from the $u$-idempotence of $\alpha'_0$ that $\big( \beta_i \big)_{i<\delta} \in Z_0$. Picking some $u$-idempotent $\big( \alpha^{(0)}_i \big)_{i<\delta} \in Z_0$, we can verify that $\big( \alpha^{(0)}_i \big)$ satisfies properties (i)--(iii) above with $k=0$.
	
	Proceeding inductively, suppose sequences $\big( \alpha^{(0)}_i \big), \big( \alpha^{(1)}_i \big), \ldots, \big( \alpha^{(k-1)}_i \big)$ have been defined as required. Let
	\begin{equation*}
		Z_k = \left\{ \big( \beta_i \big)_{i<\delta} \in \Pi_\F :\ \forall i \leq k\ \Big( \beta_i \sim \alpha^{(k-1)}_i\ \wedge\ \forall j \geq i,\ \beta_j + \halpha^{(k-1)}_i\ \sim\, \beta_j \Big) \right\}
	\end{equation*}
	
	Then, $Z_k$ is a compact $u$-semigroup by Lemma \ref{zero}. Furthermore, it is nonempty, since it contains the element $\big( \beta_i \big)_{i<\delta}$ defined by
	\begin{equation*}
		\beta_i\ \sim\ \alpha^{(k-1)}_i\ +\ \halpha^{(k-1)}_{i-1}\ +\ \h\halpha^{(k-1)}_{i-2}\ +\ \cdots\ +\ \khalpha{i}^{(k-1)}_0
	\end{equation*}
	which is $\F$-coherent by Lemma \ref{lem:seq-sum}. Since $\big( \alpha^{(k-1)}_i \big)$ satisfies properties (i)--(iii), we have that $\big( \beta_i \big)_{i<\delta} \in Z_k$. Picking some $u$-idempotent $\big( \alpha^{(k)}_i \big)_{i<\delta} \in Z_k$, we can verify that $\big( \alpha^{(k)}_i \big)$ satisfies properties (i)--(iii) above.
	
	Finally, taking
	\begin{equation*}
		\alpha_i\ \sim\ \alpha^{0}_0\ +\ \halpha^{1}_1\ +\ \h\halpha^{2}_2\ +\ \cdots\ +\ \khalpha{i}^{i}_i
	\end{equation*}
	for each $i \in \N$, we get an $\F$-Ramsey sequence $\big( \alpha_i \big)_{i<\delta} \in \Pi_\F$.
\end{proof}

\begin{remark}
	In \cite[Thm 3.8]{farahPartitionTheoremsLayered2002}, assuming rather strong conditions on $S$ and $\F$, a direct construction of an $\F$-coherent is given. Along with Theorems \ref{thm:ram1} and \ref{thm:coh->ram}, this gives sufficient conditions for a Ramsey statement on $S$ and $\F$ to be true. Unfortunately, such conditions do not apply to many examples of layered semigroups which we consider.
\end{remark}

\section{Applications}
\label{sec:appl}

\subsection{Gowers' theorem}
\label{sec:gow}

Let $\FIN$ be the Gowers semigroup of Example \ref{exm:gow-sg}, and $\F$ be the collection of generalised tetris operations on $\FIN$, as in Example \ref{exm:gow-maps}. For any $\varphi \in \g\FIN_1$, the sequence $(\alpha_j)$ defined $\alpha_j(n) = j \cdot \varphi(n)$ is $\F$-coherent. Thus, by Theorem \ref{thm:coh->ram}, $\FIN$ is $\F$-Ramsey. Applying Theorem \ref{thm:ram1} gives us the generalised Gowers' theorem of Lupini \cite[Thm 1.1]{lupiniGowersRamseyTheorem2017}:

\begin{corollary}[Lupini]\label{cor:gow}
	For any finite colouring of $\FIN$, there exists a block sequence $(f_i)_{i=1}^\infty \subseteq \FIN_n$ such that for every $k \leq n$, the following set is monochromatic:
	\begin{equation*}
		\FIN_k\ \cap\ \big\{\ \tilde{F}_1(f_{n_1}) + \cdots + \tilde{F}_\ell(f_{n_\ell})\ :\ n_1 < \cdots < n_\ell,\ \tilde{F}_1, \ldots, \tilde{F}_\ell \in \F\ \big\}
	\end{equation*}
\end{corollary}

The finite version (due to Barto\v sov\'a and Kwiatkowska \cite[Cor 2.7]{bartosovaGowersRamseyTheorem2017}) is obtained from Theorem \ref{thm:ram-fin-seq}:

\begin{corollary}[Barto\v sov\'a--Kwiatkowska]
	For any $m \in \N$ and finite colouring of $\FIN$, there exists a block sequence $(f_1,\ldots,f_m) \subseteq \FIN_n$ of length $m$ such that for every $k \leq n$, the following set is monochromatic:
	\begin{equation*}
		\FIN_k\ \cap\ \big\{\ \tilde{F}_1(f_{n_1}) + \cdots + \tilde{F}_\ell(f_{n_\ell})\ :\ n_1 < \cdots < n_\ell \leq m,\ \tilde{F}_1, \ldots, \tilde{F}_\ell \in \F\ \big\}
	\end{equation*}
\end{corollary}

Gowers' original theorem \cite[Theorem 1]{gowersLipschitzFunctionsClassical1992} is obtained by taking the subset $\F' \subseteq \F$ consisting of iterates of the \textit{tetris operation} $T(n) = \max\{ n-1, 0 \}$. Since any $\F$-coherent sequence is also $\F'$-coherent, it follows that $\FIN$ is $\F'$-Ramsey, so Theorem \ref{thm:ram1} gives:

\begin{corollary}[Gowers]
	For any finite colouring of $\FIN$, there exists a block sequence $(f_i)_{i=1}^\infty \subseteq \FIN_n$ such that for every $k \leq n$, the following set is monochromatic:
	\begin{equation*}
		\FIN_k\ \cap\ \big\{\ \tilde{T}^{(m_1)}(f_{n_1}) + \cdots + \tilde{T}^{(m_\ell)}(f_{n_\ell})\ :\ n_1 < \cdots < n_\ell;\ m_1, \ldots, m_\ell \in \N\ \big\}
	\end{equation*}
	where $\tilde{T}^{(m)}$ denotes the $m$th iterate of the tetris operation.
\end{corollary}

By identifying a set $A$ with its characteristic function $\mathbf{1}_A$, $\FIN_1$ is identified with $\Pow_\text{fin}(\N) = \{ A \subseteq \N: A \text{ finite} \}$. Then, the case $k=n=1$ gives Hindman's finite unions theorem \cite[Cor 3.3]{hindmanFiniteSumsSequences1974}:

\begin{corollary}[Hindman]\label{cor:hindman-fu}
	For any finite colouring of $\Pow_\text{fin}(\N)$, there exists a disjoint sequence $(A_i)_{i=1}^\infty$ of finite subsets of $\N$ such that the set $\{ A_{n_1} \cup \cdots \cup A_{n_\ell}:\  n_1 < \cdots < n_\ell \}$ of all finite unions of elements of $(A_i)_{i=1}^\infty$ is monochromatic.
\end{corollary}

Gowers' theorem admits many variants. For $f, g \in \FIN$, say $f < g$ if $\max(\supp(f)) < \min(\supp(g))$. Then, we can further restrict the operation of pointwise sum $f+g$ to only be defined when $f < g$. The resulting semigroup is still adequate, and the $\g \FIN_k$ are identical to before. Therefore, exactly the same argument proves the \textit{strong Gowers' theorem}:

\begin{corollary}[Strong Gowers']\label{cor:str-gow}
	For any finite colouring of $\FIN$, there exists an increasing sequence $(f_1 < f_2 < \cdots) \subseteq \FIN_n$ such that for every $k \leq n$, the following set is monochromatic:
	\begin{equation*}
		\FIN_k\ \cap\ \big\{\ \tilde{F}_1(f_{n_1}) + \cdots + \tilde{F}_\ell(f_{n_\ell})\ :\ n_1 < \cdots < n_\ell,\ \tilde{F}_1, \ldots, \tilde{F}_\ell \in \F\ \big\}
	\end{equation*}
\end{corollary}

Again, $k=n=1$ gives a strengthening of Corollary \ref{cor:hindman-fu}---we can find an increasing sequence $(A_1 < A_2 < \cdots)$ with the same property, where for $A, B \subseteq \N$, $A < B \iff \max(A) < \min(B)$.

\subsection{The Graham--Rothschild theorem}
\label{sec:gr}

Let $W$ be the Graham--Rothschild semigroup of Example \ref{exm:gr-sg}, and $\F$ be the collection of substitution maps $\tilde{w}$ on $W,$ as in Example \ref{exm:gr-maps}.

\begin{theorem}[\cite{carlsonInfinitaryExtensionGraham2006,hindmanCombiningExtensionsHalesJewett2019}]\label{thm:gr-f-coh}
	There is an $\F$-coherent $(\alpha_i) \in \Pi_W$.
\end{theorem}

As we have seen, Theorem \ref{thm:gr-f-coh} is enough to deduce a Ramsey result about $W.$ In fact, Theorems \ref{thm:ram1} and \ref{thm:ram2} give the main results from \cite{carlsonInfinitaryExtensionGraham2006}, infinitary versions of the Graham--Rothschild theorem:

\begin{corollary}[Carlson--Hindman--Strauss, Theorem 1.4]
	For any finite colouring of $W,$ there exists a sequence $(u_i)_{i=0}^\infty \in \prod_{i=0}^\infty W_i$ such that for every $k \in \N$, the following set is monochromatic:
	\begin{equation*}
		W_k\ \cap\ \big\{\ u_{n_1}[\tilde{w}_1] \ct \cdots \ct u_{n_\ell}[\tilde{w}_\ell]\ :\ n_1 < \cdots < n_\ell,\ \tilde{w}_1, \ldots, \tilde{w}_\ell \in \F\ \big\}
	\end{equation*}
\end{corollary}

\begin{corollary}[Carlson--Hindman--Strauss, Corollary 1.7]
	For any finite colouring of $W,$ there exists a sequence $(u_i)_{i=0}^\infty \subseteq W_n$ such that for every $k \leq n$, the following set is monochromatic:
	\begin{equation*}
		W_k\ \cap\ \big\{\ u_{n_1}[\tilde{w}_1] \ct \cdots \ct u_{n_\ell}[\tilde{w}_\ell]\ :\ n_1 < \cdots < n_\ell,\ \tilde{w}_1, \ldots, \tilde{w}_\ell \in \F\ \big\}
	\end{equation*}
\end{corollary}

Corollary \ref{cor:ram-fin} gives Graham and Rothschild's original theorem \cite{grahamRamseyTheoremParameter1971}:

\begin{corollary}[Graham--Rothschild]\label{cor:gr}
	For every $m \geq k$ and every finite colouring of $W_k$, there exists $u \in W_m$ such that the set $\{ u[\tilde{w}] : \tilde{w}{\upharpoonright_m} \in W_k \}$ is monochromatic.
\end{corollary}

The case $k=0$, $m=1$ of Corollary \ref{cor:gr} gives the famed Hales--Jewett theorem:

\begin{corollary}[Hales--Jewett]\label{cor:hj}
	For every finite colouring of $W_0 = A^{<\omega}$, there is a variable word $u \in W_1$ such that $\{ u[a] : a \in A \}$ is monochromatic.
\end{corollary}

Corollary \ref{cor:hj} is known to imply van der Waerden's theorem \cite[p38]{grahamRamseyTheory1990}:

\begin{corollary}[van der Waerden]
	For every $k \in \N$ and finite colouring of $\N$, there is $a, d \in \N$ such that the arithmetic progression $\{ a + jd: 0 \leq j < k \}$ is monochromatic.
\end{corollary}

\begin{remark}\label{rmk:gr-ram}
	A construction satisfying Theorem \ref{thm:gr-f-coh} is given by Carlson, Hindman and Strauss \cite[Thm 2.12]{carlsonInfinitaryExtensionGraham2006}, and by Hindman, Strauss and Zamboni \cite{hindmanCombiningExtensionsHalesJewett2019}. In fact, in both cases, they directly construct an $\F$-Ramsey, all of whose elements are $\preccurlyeq$-minimal. Both arguments used to construct such a sequence are extraordinarily involved, long and complicated. The author believes that an $\F$-coherent could be constructed by a significantly simpler argument, from which an $\F$-Ramsey could be constructed via the general framework laid out in this paper, but has not been able to significantly simplify their argument, nor find an alternative method.
\end{remark}

\begin{remark}\label{rmk:str-vw}
	Nonetheless, if we restrict to the subset $\F' \subseteq \F$ of \textit{strong variable words}, i.e.\ those where all variables appear in order, a simple nonstandard argument can construct an $\F'$-coherent. We rely on the following result from the theory of compact semitopological semigroups:
\end{remark}

\begin{lemma}[{\cite[Lemma 2.3]{todorcevicIntroductionRamseySpaces2010}}]\label{lem:id-below}
	If $\alpha \in \hK$ is $u$-idempotent, and $I \subseteq \hK$ is a closed left ideal, then there is $u$-idempotent $\beta \in I + \halpha$ such that $\beta \preccurlyeq \alpha$.
\end{lemma}

\begin{lemma}[{\cite[Lemma 7.9]{dinassoNonstandardMethodsRamsey2019}}]\label{lem:hj-ns}
	There are $u$-idempotents $\nu \in \h W_0$, $\om \in \h W_1$ such that $\om \preccurlyeq \nu$ and $\om[a] \sim \nu$ for every $a \in A$.
\end{lemma}

\begin{proof}
	Let $T \defeq W_0 \cup W_1$. Pick any $\preccurlyeq$-minimal $u$-idempotent $\nu \in \h W_0$ by Lemma \ref{cor:u-id-min}. Note that $\h W_1$ is a closed left ideal of $\h T$, so by Lemma \ref{lem:id-below}, pick $u$-idempotent $\om \in \h W_1 \ct \h \nu \subseteq \h W_1$ such that $\om \preccurlyeq \nu$. Now, for any $a \in A$, the substitution map $v \mapsto v[a]$ is a homomorphism $T \to W_0$, hence $\om[a] \preccurlyeq \nu[a] = \nu$ since $\nu \in \h W_0$. Since $\nu$ is $\preccurlyeq$-minimal, it follows that $\om[a] \sim \nu$.
\end{proof}

\begin{theorem}\label{thm:gr-str-f-coh}
	There is an $\F'$-coherent $(\alpha_i) \in \Pi_W$.
\end{theorem}

\begin{proof}
	Let $\nu$, $\om$ be as in Lemma \ref{lem:hj-ns}, and define $\alpha_i$ by $\alpha_0 = \nu$, $\alpha_1 = \om$, $\alpha_i \sim \om[x_1] \ct \h\om[x_2] \ct \cdots \ct \kh{i}\om[x_i]$ for $i \geq 2$. By $u$-idempotence of $\nu$, $\om$ and the fact that $\om \preccurlyeq \nu$, $(\alpha_i)$ is $\F'$-coherent.
\end{proof}

\subsection{Galvin--Glazer and Hindman's theorem}
\label{sec:gg}

Let $S$ be an adequate partial semigroup such that $s+s$ is never defined. $S$ is trivially layered by the layering map $\ell: S \to \N$, $s \mapsto 0$. Let $\F$ consist of only the identity map $\id: S \to S$, which is a regressive map. Then, any $\alpha \in \hS$ is trivially $\F$-coherent, thus $S$ is $\F$-Ramsey. Applying Theorem \ref{thm:ram1} with $n=0$ gives the \textit{Galvin--Glazer theorem} (see \cite[Thm 2.20]{todorcevicIntroductionRamseySpaces2010} or \cite{comfortUltrafiltersOldNew1977}):

\begin{corollary}[Galvin--Glazer]
	For any finite colouring of $S$, there exists an infinite sequence $(x_i)_{i=0} \subseteq S$ of distinct elements such that the set
	\begin{equation*}
		\big\{\ x_{n_0} + \cdots + x_{n_{\ell-1}}\ :\ n_0 < \cdots < n_{\ell-1},\ \big\}
	\end{equation*}
	is $c$-monochromatic.
\end{corollary}

The special case $S = \N$ and $A = \{ x_i : i \in \N \}$ gives us Hindman's finite sums theorem \cite[Thm 3.1]{hindmanFiniteSumsSequences1974}:

\begin{corollary}[Hindman]
	For any finite colouring $c: \N \to r$, there exists an infinite set $A \subseteq \N$, such that the set $\FS(A)$ of all finite, nonrepeating sums from $A$ is $c$-monochromatic.
\end{corollary}

\section{Located variable words}
\label{sec:lv-words}

In this section, we give a simple nonstandard proof of Bergelson, Blass and Hindman's partition theorem for \textit{located variable words} \cite[Thm 4.1]{bergelsonPartitionTheoremsSpaces1994}, using Theorem \ref{thm:gr-f-coh} and the general framework of layered semigroups developed herein. In fact, we obtain a multivariable generalisation of that result, related to a result of Solecki \cite[\S 4.3.1]{soleckiMonoidActionsUltrafilter2019}.
%In this section, we give a simple nonstandard proof of Solecki's Ramsey theorem for \textit{located variable words} \cite[\S 4.3.1]{soleckiMonoidActionsUltrafilter2019}, using Theorem \ref{thm:gr-f-coh} and the general framework of layered semigroups developed herein.

Fix a finite alphabet $A$. A \textit{located word} over $A$ \cite[\S4.2]{lupiniTrees} is a partial function $w: \N \to A$, where $\dom(w)$ is finite and nonempty. In the vein of Example \ref{exm:gr-sg}, a \textit{located $k$-parameter word} over $A$ is a partial function $w: \N \to A \cup \{ x_1, \ldots, x_k \}$, where $\dom(w)$ is finite, all the $x_i$ appear, and in increasing order.

For each $k \in \N$, we let $L_k$ be the set of all located $k$-parameter words over $A$ (where $L_0$ simply contains located words). Then, $L = \bigcup_{i=0}^\infty L_i$ is a layered semigroup, with the operation $+$ defined
\begin{equation*}
	(w+v)(n) = \begin{cases}
		w(n) & \text{ if } n \in \dom(w) \\
		v(n) & \text{ if } n \in \dom(v) \\
		\text{undefined } & \text{ otherwise} \\
	\end{cases}
\end{equation*}
only when $\dom(w) \cap \dom(v) = \varnothing$. Notice that $L$ is an adequate partial semigroup, and the $\gL_m$ are exactly the \textit{cofinite} located $m$-variable words $\alpha$---those whose domains $\dom(\alpha) \subseteq \h\N$ are disjoint from $\N$.

As in Example \ref{exm:gr-maps}, substitution maps $u \mapsto u[\tilde{w}]$ are defined by infinite (total/nonlocated) variable words $\tilde{w}$, in much the same way:
\begin{equation*}
	u[\tilde{w}](n) = \begin{cases}
		u(n) & \text{ if } n \in \dom(u), u(n) \neq x_i \text{ for all } i \\
		\tilde{w}_i & \text{ if } n \in \dom(u), u(n) = x_i \text{ for some } i \\
		\text{undefined } & \text{ if } n \notin \dom(u) \\
	\end{cases}
\end{equation*}
In fact, we can naturally identify the Graham--Rothschild semigroup $W$ as a subset of $L$, and the above definition extends the substitution maps on $W.$ As before, all such substitution maps are regressive, and the collection $\F$ of all of them is locally finite, and closed under composition.

Assuming the existence of an $\F$-coherent in $\Pi_W$, we can construct an $\F$-coherent in $\Pi_L$ by a simple nonstandard argument:

\begin{lemma}\label{thm:l-f-coh}
	$L$ has an $\F$-coherent.
\end{lemma}

\begin{proof}
	Pick an $\F_W$-coherent sequence $(\alpha_i)_{i=0}^\infty$ in the Graham--Rothschild semigroup. Note that all substitution maps preserve the length of words, thus each $\alpha_i$ must have infinite length. If any $\alpha_i$ was finite, $\F$-coherence and Proposition \ref{prop:dinasso}.\ref{prop:u-eq-injec} would imply all the $\alpha_i$ are finite, and of the same length $\abs{\alpha_i}=n$. But then $\alpha_{n+1}$ cannot exist, since there are no $(n+1)$-parameter words of length $n$.
	
	Define $g: L \to L$ as the map which takes the ``second half'' of any located variable word:
	\begin{equation*}
		g(w)(n) = \begin{cases}
			w(n) & n \geq \floor{ \big( \abs{w}/2 \big) } \\
			\text{undefined} & n < \floor{ \big( \abs{w}/2 \big) }
		\end{cases}
	\end{equation*}
	Note that $g$ commutes with every substitution map, i.e.\ $g \big( u[\tilde{w}] \big) = g(u)[\tilde{w}]$ for all $u \in L$, $\tilde{w} \in \F$. By transfer, this is also true for nonstandard words $\alpha \in \h\hspace{-0.1em}L$.
	
	Now, define $\beta_i = g(\alpha_i)$ for each $i \in \N$. Since the lengths $\abs{\alpha_i} \in \h\N \setminus \N$ are infinite, it follows that each $\beta_i$ is undefined up to $\floor{ \big( \abs{\alpha_i}/2\big)} \in \h\N \setminus \N$, whence $\beta_i \in \gL_i$. Thus, $(\beta_i)_{i=0}^\infty \in \Pi_L$. The $\F$-coherence of $(\beta_i)$ follows from Proposition \ref{prop:dinasso}.\ref{prop:f-alpha}---for all $i \in \N$ and $\tilde{w} \in \F$, there is $j \leq i$ such that
	\begin{equation*}
		\beta_i[\tilde{w}]\ =\ g(\alpha_i)[\tilde{w}]\ =\ g \big( \alpha_i[\tilde{w}] \big)\ \sim\ g \big( \alpha_j \big)\ =\ \beta_j \qedhere%\\[-6mm]
	\end{equation*}
	
%	Naturally, we can consider a word $\alpha$ of length $\xi \in \h\N$ as a function $\alpha: [0,\xi) \to A \cup X$. We ``cofinitize'' $(\alpha_i)$ by defining $(\alpha'_i) \in \Pi_L$ by
%	\begin{equation*}
%		\alpha'_i(\xi) = \begin{cases}
%			\text{undefined } & \text{ if } \xi \in \N \\
%			\alpha_i(\xi) & \text{ if } \xi \in \h\N \setminus \N \\
%		\end{cases}
%	\end{equation*}
\end{proof}

It follows by Theorem \ref{thm:coh->ram} that $L$ is $\F$-Ramsey. Applying Theorems \ref{thm:ram1} gives a multivariable generalisation of the Bergelson--Blass--Hindman theorem on located words \cite[Thm 4.1]{bergelsonPartitionTheoremsSpaces1994}:

\begin{corollary}\label{cor:solecki}
	For any $n$ and finite colouring of $L$, there exists a block sequence $(u_i)_{i=1}^\infty \subseteq L_n$ such that for every $k \leq n$, the following set is monochromatic:
	\begin{equation*}
		L_k\ \cap\ \big\{\ u_{n_1}[\tilde{w}_1] + \cdots + u_{n_\ell}[\tilde{w}_\ell]\ :\ n_1 < \cdots < n_\ell,\ \tilde{w}_1, \ldots, \tilde{w}_\ell \in \F\ \big\}
	\end{equation*}
\end{corollary}

The original theorem of Bergelson, Blass and Hindman is the case $n=1$. If we instead take Corollary \ref{cor:solecki} with $n=k=1$ and $A=\varnothing$, we again obtain Hindman's finite unions theorem (Corollary \ref{cor:hindman-fu}).

\begin{remark}
	The proof of Lemma \ref{thm:l-f-coh} is also valid for the subcollection $\F' \subseteq \F$ of substitution maps corresponding to \textit{strong variable words} (see Remark \ref{rmk:str-vw}), so along with Theorem \ref{thm:gr-str-f-coh}, this gives a short nonstandard proof of the weaker version of Corollary \ref{cor:solecki} for strong variable words.
\end{remark}

\begin{remark}
	As with Corollary \ref{cor:str-gow}, we can also restrict the operation $w+v$ to only be defined when $w < v$, i.e.\ $\dom(w) < \dom(v)$. The proof then goes through unchanged, and we get a stronger version of Corollary \ref{cor:solecki}, where the block sequence $(u_1 < u_2 < \cdots) \subseteq L_n$ can be taken to be increasing.
\end{remark}

\subsection{A common generalisation}
\label{sec:gowers-bbh}

It is notable that Gowers' theorem and the Bergelson--Blass--Hindman theorem \cite[Thm 4.1]{bergelsonPartitionTheoremsSpaces1994} both generalise the finite unions theorem of Hindman \cite[Cor 3.3]{hindmanFiniteSumsSequences1974}. Here, we present a common generalisation of both theorems \cite{lupiniGowersRamseyTheorem2017,lupiniTrees}, and prove it using the general framework of layered semigroups previously described.

Let $A$ be a finite alphabet, and $X = \{ x_0, x_1, x_2, \ldots \}$ be a countably infinite set of variables. Note that, compared to the Graham--Rothschild case, we have added an extra variable symbol $x_0$, which we interpret to mean ``undefined''. Let $\FIN^A$ be the set of (total) $f: \N \to A \cup X$ which are eventually constant and equal to $x_0$.

For $f,g \in \FIN^A$, the sum $f+g$ is defined iff for every $n \in \N$, at least one of $f(n)$, $g(n)$ is $x_0$. In this case, $f+g$ is defined as
\begin{equation*}
	(f+g)(n) = \begin{cases}
		f(n) & \text{if } f(n) \neq x_0 \\
		g(n) & \text{if } f(n) = x_0
	\end{cases}
\end{equation*}
Under this operation, and the layering map $\ell(f) = \max\{ k: f(n)=x_k \text{ for some } n \}$, $\FIN^A$ is a commutative, adequate, partial layered semigroup.

Call a (total) map $F: \N \to A \cup X$ \textit{strong} if $F(0)=x_0$, and all $x_i$ appear as values of $F$ in increasing order (i.e.\ whenever $F(n)=x_i$, $F(m)=x_j$ for $n \leq m$, then $i \leq j$). The strong maps act on $\FIN^A$ by ``composition'':
\begin{equation*}
	\tilde{F}[f](n) = \begin{cases}
		F(k) & \text{if } f(n) = x_k \text{ for some } k \geq 1 \\
		f(n) & \text{if } f(n) \in A \cup \{ x_0 \}
	\end{cases}
\end{equation*}
Every strong $\tilde{F}$ is a regressive map, and the collection $\F$ of all such strong maps is locally finite, and closed under composition.

Identifying ``undefined'' with the symbol $x_0$, the semigroup $L^A$ of located words over $A$ is a subset of $\FIN^A$. As this containment is strict, we need to strengthen Theorem \ref{thm:ram1} to ensure that the sequence can be found inside a specified subsemigroup $T \subseteq \FIN^A$. Henceforth, we assume $T$ is layerwise adequate, and closed under all strong maps. We define $\Pi_T \defeq \prod_{i=0}^\infty \g T_i$, and say that $T$ is $\F$-Ramsey if $\Pi_T$ contains an $\F$-Ramsey element. Theorem \ref{thm:ram1} admits the following strengthening:

\begin{theorem}\label{thm:fin-a}
	Suppose $T \subseteq \FIN^A$ is $\F$-Ramsey, and $n \in \N$. Then, for any finite colouring of $\FIN^A$, there exists a block sequence $(f_i)_{i=1}^\infty \subseteq T_n \defeq T \cap \FIN^A_n$ such that for every $k \leq n$, the following set is monochromatic:
	\begin{equation*}
		\FIN^A_k\ \cap\ \big\{\ \tilde{F}_1(f_{n_1}) + \cdots + \tilde{F}_\ell(f_{n_\ell})\ :\ n_1 < \cdots < n_\ell,\ \tilde{F}_1, \ldots, \tilde{F}_\ell \in \F\ \big\}
	\end{equation*}
\end{theorem}

\begin{proof}
	As for Theorem \ref{thm:ram1}, but replace all occurrences of $S_n$ with $T_n$.
\end{proof}

The question remains---when is $T$ $\F$-Ramsey?

\begin{definition}
	A subsemigroup $T \subseteq \FIN^A$ is \textit{complete} if for any $\alpha \in \g T_1$ and $k \in \N$, we have $\alpha + \halpha[x_2] + \cdots + \khalpha{(k-1)}[x_k] \in \kh{k}T$, where $\alpha[b]$ is the function obtained by setting $\alpha[b](n) = b$ whenever $\alpha(n) = x_1$.
\end{definition}

\begin{lemma}\label{lem:compl}
	Every complete $T \subseteq \FIN^A$ is $\F$-Ramsey.
\end{lemma}

\begin{proof}
	Pick $u$-idempotent $\nu \in \g T_0$, $\om \in \g T_1$ such that $\om \preccurlyeq \nu$ and $\om[a] \sim \nu$ for every $a \in A$ (see Lemma \ref{lem:hj-ns}). Since $T$ is complete, the argument in Theorem \ref{thm:gr-str-f-coh} works, giving an $\F$-coherent in $\Pi_T$. Theorem \ref{thm:coh->ram} now shows $T$ is $\F$-Ramsey, by restricting to $\Pi_T$ rather than $\Pi_S$.
\end{proof}

Note that $\FIN^A$ is trivially complete, hence it is $\F$-Ramsey. Applying Theorem \ref{thm:fin-a} with $A = \varnothing$ and $T = \FIN^\varnothing = \FIN$ gives Gowers' theorem (Corollary \ref{cor:gow}, where we interpret each variable symbol $x_k$ as the number $k \in \N$). $L_A \subseteq \FIN^A$ is also complete, and Theorem \ref{thm:fin-a} with $T = L_A$ gives the Bergelson--Blass--Hindman theorem for strong variable words (Corollary \ref{cor:solecki}, interpreting $x_0$ to mean ``undefined'').

\begin{remark}
	Theorem \ref{thm:fin-a} doesn't recover the full strength of Corollary \ref{cor:solecki}---we could attempt to do so by considering a wider class $\F^+$ of regressive maps $F: \N \to A \cup X$, where only the \textit{first} occurrences of each $x_i$ must appear in increasing order. It is unclear whether $\FIN^A$ is $\F^+$-Ramsey. If it were, we would expect that the argument to construct an $\F^+$-coherent would be more complicated, as per Remark \ref{rmk:gr-ram}. This would give an even more generalised version of Gowers' theorem, where some layers of a function $f \in \FIN$ can be ``reversed''.
\end{remark}

In the same way as Lupini \cite[\S 3]{lupiniGowersRamseyTheorem2017} and Dodos--Panellopoulos \cite[Thm 2.21]{dodosRamseyTheoryProduct2016}, we can generalise Theorem \ref{thm:fin-a} to cover the Milliken--Taylor theorem. For $m \in \N$, let $\FIN^{A[m]}$ be the collection of block sequences $(f_1,\ldots,f_m)$ of elements of $\FIN^A$.

\begin{theorem}
	Suppose $T \subseteq \FIN^A$ is $\F$-Ramsey, and $n \in \N$. Then, for any finite colouring of $\FIN^A$, there exists a block sequence $(x_i)_{i=1}^\infty \subseteq T_n$ such that for every $k \leq n$, the following set is monochromatic:
	\begin{align*}
		\bigg\{\ \Big( \tilde{F}_1(x_{n_1}) + &\cdots + \tilde{F}_{\ell_1}(x_{n_{\ell_1}}),\ \ldots,\ \tilde{F}_{\ell_{m-1}+1}(x_{n_{\ell_{m-1}+1}}) + \cdots + \tilde{F}_{\ell_m}(x_{n_{\ell_m}}) \Big) :\\
		&0<\ell_1 < \cdots < \ell_m; \ n_1 < \cdots < n_{\ell_m};\ \tilde{F}_1, \ldots, \tilde{F}_{\ell_m} \in \F\ \bigg\}\ \cap\ \FIN^{A[m]}_k
	\end{align*}
\end{theorem}

\begin{proof}
	Identical to \cite[Thm 8.16]{dinassoNonstandardMethodsRamsey2019}.
\end{proof}

Theorem \ref{thm:fin-a} is the case $m=1$, while the Milliken--Taylor theorem is the case $A=\varnothing$, $T = \FIN^\varnothing = \FIN$, $n=1$.

\section{Closing remarks}

This paper outlines a general framework for proving Ramsey-type results about layered semigroups. We have seen in \S\ref{sec:appl} and \S\ref{sec:lv-words} that many old results in Ramsey theory can be expressed and proven via our framework. Since layered semigroups form a rich class of structures, we also expect that one could use this framework to prove new Ramsey-theoretic results. Unfortunately, we are not aware of any immediate applications to known open problems.

\begin{problem}
    Use the framework of this paper to prove new Ramsey-theoretic results of interest.
\end{problem}

We note that there are still Ramsey statements about semigroups which we have not been able to recover. In particular, we saw that van der Waerden's theorem is a corollary of the Graham--Rothschild theorem, which can be proven in our framework, but a direct formulation remains elusive. We could attempt to obtain it from Theorem \ref{thm:ram2} by letting $S = \N^+ \times \{0,1\}$ and $\ell: (n,i) \mapsto i$. Then, consider maps $f_c: S \to S_0$, where $f_c: (n,0) \mapsto (n,0)$, $(m,1) \mapsto (cm,0)$ for all $c$ up to some fixed $k$. Taking the operation $+$ on $S_0$ to be usual addition, Theorem \ref{thm:ram2} obtains the structure of van der Waerden's theorem. However, it may not be possible to extend $+$ to $S$ so that $(S,\ell)$ is layered and the $f_c$ are homomorphisms.

\begin{question}
	Is there a layered semigroup $S$, and regressive maps $\F$ on $S$, such that Theorem \ref{thm:ram1} or Theorem \ref{thm:ram2} reduce to van der Waerden's theorem?
\end{question}

Another result notably missing is Ramsey's theorem itself. The most sensible attempt seems to be taking $S = \Pow_\text{fin}(\N) = \{ F \subseteq \N: F \text{ finite} \}$, and layering it by $\ell: F \mapsto \abs{F}$. Then, given a sequence $(F_i)_{i=1}^\infty$ of finite sets, an infinite homogeneous set $B$ could be obtained by $B = \bigcup_{i=1}^\infty F_i$. The challenge is how to define a semigroup operation and regressive maps that allow us to generate any finite subset of $B$.

\begin{question}
	Is there a layered semigroup $S$, and regressive maps $\F$ on $S$, such that Theorem \ref{thm:ram1} or Theorem \ref{thm:ram2} reduce to Ramsey's theorem?
\end{question}

Definition \ref{defn:regr} imposed strong conditions on the types of maps considered, particularly conditions (iii) and (iv). While all examples considered satisfied all these conditions, perhaps a different argument is possible which doesn't require these conditions. This would require a new construction in the proof, since conditions (iii) and (iv) are necessary to Lemma \ref{lem:seq-sum}.

\begin{question}
	Can Theorem \ref{thm:coh->ram} be proven when conditions (iii) and/or (iv) are weakened or removed from Definition \ref{defn:regr}?
\end{question}

Our framework allows us to deduce Ramsey statements about $(S,\F)$ from the existence of an $\F$-coherent in $S$. However, we have not found a general way to construct $\F$-coherents in arbitrary layered semigroups $S$, or ensure they exist. The most general construction we have given of an $\F$-coherent is Lemma \ref{lem:compl}, but this still depends crucially on the structure of $\FIN^A$. Farah--Hindman--McLeod \cite[Thm 3.8]{farahPartitionTheoremsLayered2002} give a fairly general construction, but it requires very strong conditions on $\F$, which do not hold in many natural examples.

\begin{problem}
	Find a general way to construct an $\F$-coherent in a layered semigroup $S$, making as few assumptions about $S$ and $\F$ as possible.
\end{problem}

We did not present a construction of an $\F$-coherent in the Graham--Rothschild semigroup $W,$ instead deferring to a result of Carlson--Hindman--Strauss \cite[Thm 2.12]{carlsonInfinitaryExtensionGraham2006}. They directly construct a $\preccurlyeq$-minimal $\F$-Ramsey, and as a result, their argument is extremely complicated. Morally, if we only need to construct an $\F$-coherent, there should be a simpler argument---then our framework would imply the Graham--Rothschild theorem.

\begin{problem}
	Find a simpler construction of an $\F$-coherent in the Graham--Rothschild semigroup $W.$
\end{problem}

Our results generalise many Ramsey-type or partition results on layered semigroups. However, there is another thread of Ramsey theory dealing with \textit{density} results, having the form that any set $A \subseteq \M$ of positive ``density'' contains a substructure $\Nn \subseteq \M$ with certain properties. Often, we take $\M = \N$ and interpret density to mean \textit{upper density}; $d(A) = \limsup_{n \to \infty} \abs{A \cap [0,n)} / n$ for each $A \subseteq N$. For example, Szemer\'edi's theorem is the density version of van der Waerden's theorem:

\begin{theorem}[Szemer\'edi]
	For any $k \in \N$ and $A \subseteq \N$ with $d(A)>0$, there are $a,d \in \N^+$ such that the arithmetic progression $\{ a+cd: 0 \leq c < k \} \subseteq A$.
\end{theorem}

By defining a suitable notion of density on $W,$ Furstenberg and Katznelson obtained a density version of the Hales--Jewett theorem \cite{furstenbergDensityVersionHales1991}. Nonstandard methods have also been applied successfully to prove density-type results \cite[Part III]{dinassoNonstandardMethodsRamsey2019}. Thus, it may be possible to develop a similar framework for proving density theorems in layered semigroups. This would require a suitable notion of density on layered semigroups---McLeod has already generalised some combinatorial notions of size in $\N$ to this setting \cite{mcleodNotionsSizePartial2000}.

\begin{problem}
	Develop an analogous framework for proving density results on layered semigroups.
\end{problem}

Variations on the basic structure of Theorem \ref{thm:ram1} also appear in Ramsey theory. For example, take the Graham--Rothschild semigroup $W,$ and let $LV \subseteq W$ be the subset consisting of \textit{left-variable words}---those whose first character is the variable $x_1$. The \textit{Hales--Jewett theorem for left-variable words} states:

\begin{theorem}[{\cite[Theorem 2.37]{todorcevicIntroductionRamseySpaces2010}}]
	For any finite colouring of $W_0 = A^{<\omega}$, there exists a word $w \in W_0$ and a sequence $(u_i)_{i=0}^\infty \subseteq LV \cap W_1$ of left-variable words such that the following set is monochromatic:
	\begin{equation*}
		\big\{\ w \ct u_{n_1}[a_1] \ct \cdots \ct u_{n_\ell}[a_\ell]\ :\ n_1 < \cdots < n_\ell,\ a_1, \ldots a_\ell \in A\ \big\}
	\end{equation*}
\end{theorem}

Effectively, we ensure that all products of subsequences have the same nonvariable part. $LV \subseteq W$ is a right ideal, so we could generalise this idea to \textit{right layered semigroups} $(S,\ell,R)$, with $R \subseteq S$ a distinguished right ideal meeting every layer except $S_0$. Then, we instead consider sequences in $\Lambda_S \defeq \hS_0 \times \prod_{1 \leq i < \delta} \h\hspace{-0.15em} R_i$, and define analogous notions of right $\F$-coherent and right $\F$-Ramsey sequences. Further investigation is required to see if the proofs of Theorems \ref{thm:ram1} and \ref{thm:coh->ram} translate to this setting. Problems may arise in translating the proof of Theorem \ref{thm:coh->ram}, if we require conditions such as $\kappa + \hf(\alpha_i) \sim \kappa + \halpha_j$, which are not continuous in $\alpha_i$.

\begin{problem}
	Translate our ideas to the setting of right layered semigroups.
\end{problem}

\section*{Acknowledgements}

Thanks go to:
\begin{itemize}
    \item Martino Lupini, for teaching me all I know about Ramsey theory and nonstandard analysis, and for first noticing the abstract connection between Gowers' theorem and the Graham--Rothschild theorem, eventually leading to the ideas in this paper.
    
    \item Valentino Vito, for thoroughly reading the first draft, and providing invaluable suggestions and corrections.
    
    \item The anonymous reviewer, for their thorough peer review and useful comments.
\end{itemize}

\end{document}